\documentclass[11pt]{amsart}
\usepackage{amsmath}
\usepackage{appendix}
\usepackage{amsthm}
\usepackage{amscd,amssymb}
\usepackage[all]{xy}
\usepackage{color}
\usepackage{appendix}
\usepackage{verbatim}
\usepackage{pdfsync}
\usepackage{mathrsfs}

 \newtheorem{thm}{Theorem}[section]
 \newtheorem{lemma}[thm]{Lemma}

\theoremstyle{definition}

 \newtheorem{defn}[thm]{Definition}
  \newtheorem{ex}[thm]{Examples}

\newcommand{\N}{\mathbb{N}}

\def\E{{\text{E}}}
\def\oI{\mathscr{I}}
\newcommand{\g}{\Gamma}
\def\S{{\text{S}}}
\def\D{{\text{D}}}

\def\Hom{{\text{Hom}}}
\def\U{{\text{U}}}
\def \cU{\mathscr{U}}

\def\cC{{\mathscr{C}}}
\def\sB{{\mathscr{B}}}
\def\sA{{\mathscr{A}}}
\def\cB{{\mathcal{B}}}
\def\cA{{\mathscr{A}}}

\def\cS{\mathscr{S}}
\def\cTg{\mathscr{T}_G}

\def\cT{\mathscr{T}}
\def\cW{\mathscr{W}}
\def\cR{\mathscr{R}}
\def\cF{\mathscr{F}}
\def\cV{\mathscr{V}}
\def\cD{\mathscr{D}}
\def\cA{\mathcal{A}}
\def\Gcat{\text{$G$Cat}}
\def\u#1{\underline{#1}}
\def\sS{\underline{S}}
\def\h#1{\hat{#1}}
\def\b#1{\textbf{#1}}

\def\cI{\mathcal{I}}

\def\cU{\mathcal{U}}

\def\cA{\mathscr{A}}
\def\cU{\mathscr{U}}
\def\Aut{\text{Aut}}
\begin{document}
%\tableofcontents

\title{A  short treatise on  Equivariant $\g$-spaces }
\author{Rekha Santhanam}
\begin{abstract}
Equivariant $\g$-spaces model equivariant infinite loop spaces. In this article, we show that there exists a connective Quillen equivalence between the category of equivariant $\g$-spaces and the category of orthogonal spectra. \end{abstract}
\maketitle
\section{Introduction}
 We begin this article by defining a level and stable model structures on the category of  equivariant $\g$-spaces. This is called the projective level  and stable model structure in \cite{Ostermayr}. Ostermayr\cite{Ostermayr}, shows that there is a connective Quillen equivalence with the category of equivariant symmetric spectra with stable model structure. The setup is in the category of simplicial sets. In this article, we show that there is a connective Quillen equivalence between the category of equivariant $\g$-spaces with projective stable model structure and the category of equivariant orthogonal spectra with stable structure. By spaces, we mean the category of compactly generated Hausdorff topological spaces.   We conclude with remarks  on what is  known so far about $\Gamma$-$G$-categories.
%give a proof of the equivariant Barratt-Priddy-Quillen theorem.  Lastly we give a comparison of this machinery with  equivariant $\E_\infty$ categories as defined in  \cite{GuillouMay}.

Throughout this article we will assume that $G$ is a finite group. 
\section{Equivariant $\g$-spaces }
Let $G \cT$ denote the category of based $G$-topological spaces with continuous $G$-maps. Let $\cTg$  denote the $G$-enriched  category of $G$-topological spaces with the set of morphisms given by all maps and $G$-conjugation action. Let $\g$ be the category of finite pointed sets, $G\g$ denote the category of finite pointed $G$-sets and $\g_G$ the $G$-enriched category of finite pointed $G$-sets with the set of morphisms being all set maps with $G$-conjugation action. Let $G \cW$ denote the category of $G$-CW complexes with continuous $G$-maps.

Define an equivariant $\g$-space $X$ to be a covariant functor from the category of $X: \g \to G\cT$ such that $X(\b 0)$ is a point.  Shimakawa \cite{Shimakawa} originally defines an equivariant $\g$-space as a $G$-functor from $\g_G \to \cT_G$.  As observed in \cite{Shimakawa2} these two categories are equivalent. 

Let $\E$ denote the prolongation functor from $\g[G\cT] \to G\g[G\cT]$. A special equivariant $\g$-space is an equivariant $\g$-space such that for every $G$-set $S$, the map 
$$ \E X(S) \to G\cT(S,\E X(\b1))$$ is a $G$-equivalence. 

Note that  if $X$ is special then $\pi_0(X(\b 1)^H)$ is a monoid via the map $ X(\b 1)^H \times X(\b 1)^H  \xleftarrow{\simeq} X(\b 2)^H \xrightarrow{X(\Delta)} X(\b 1)^H$ where, $\Delta : \b 2 \to \b 1$ is defined by $\Delta(\{1,2\})=\{1\}$. A special equivariant $\g$-space is said to be very-special if $X(\b 1)^H$ is group-like for all $H <G$. 

A map of $G$-spaces $X \to Y$ is said to be $G$-weak equivalence ($G$-fibration) if for every $H<G$ the map $X^H \to Y^H$ is a weak equivalence (Serre fibration). 

\begin{thm}
The category $G\g[G\cT]$ is a cofibrantly generated model category with level model structure where, 
\begin{itemize}
\item weak equivalences are levelwise $G$-weak equivalences. 
\item fibrations are levelwise $G$-fibrations and,
\item cofibrations are q-cofibrations, that is, the morphisms with right lifting property with respect to all acyclic fibrations. 
\end{itemize}
The generating  acylic cofibrations are given by $J= \{G\g_{S} \wedge G/H \times \D^n \to   G\g_{S} \wedge G/H \times \D^n \times I \ | \ n \in \N, H <G \} $ and the generating cofibrations are given by $I=\{G\g_{S} \wedge G/H \times \S^{n-1} \to   G\g_{S} \wedge G/H \times \D^n \ | \ n \in \N, H <G  \} $.
\end{thm}
\begin{proof}
Follows from \cite{RS}[Thm 5.3]
\end{proof}
Let  $X: G \g \to G \cT$   be a functor.  Let $\E X$ denote the prolongation of $X$ to $G\cW$.  Let $\cU$ be a $G$-universe. Then  $(\E X(S^V))_{V \in \cU}$ defines an equivariant spectrum.  
A morphism $X \to Y$ in $G\g[G\cT]$ is said to be a stable $G$-equivalence if  $\E X(S^V) \to \E Y(S^V)$ is a $\pi^H_\ast$-isomorphism for all $H<G$.

\begin{thm}
Localizing with respect to all stable equivalences gives a cofibrantly generated  stable model structure on $G\g[G\cT]$ where,
\begin{itemize}
\item weak equivalences are stable equivalences. 
\item cofibrations are q-cofibrations and,
\item fibrations are q-fibrations, that is, morphisms with right lifting property with respect to all acyclic q-cofibrations 
\end{itemize}
\end{thm}
\begin{thm} 
Let $G\g[G \cT]$ have the stable model structure. Then the  category $\g[G\cT]$ forms a  cofibrantly generated model category with the following model structure;
\begin{itemize}
\item a morphism $X \to Y$  in $\g[G\cT]$ is a weak equivalence if  the morphism $\E X \to \E Y $ is a weak equivalence  in $G \g [G \cT]$,
\item a morphism $X \to Y$  in $\g[G\cT]$ is a fibration if  the morphism $\E X \to \E Y $ is a fibration in $G \g [G \cT]$ and,
\item a morphism $X \to Y$  in $\g[G\cT]$ is a cofibration if  it has left lifting property with respect to all acyclic fibrations.
\end{itemize}
The fibrant objects will be very-special equivariant $\g$-spaces.
\end{thm}
\begin{proof} 
Let $\U $ denote the forgetful functor adjoint to $E$. The collection of maps $UI$ and $UJ$ satisfy the cofibration hypothesis for $\sA = \cC=\g[G\cT]$. The proof follows from Theorem \ref{lem}.
\end{proof}

Alternately we could describe the stable model structure on the category of equivariant $\g$-spaces by localizing the projective level model structure on it with respect to stable equivalences \cite{RS}[Lemma  A.4] as described in \cite{Ostermayr}. %We present it in the above manner for the ease of proof of Theorem \ref{connective}

\section{Equivariant Orthogonal spectra}
\begin{defn}\cite{Schwede}
An $G$-orthogonal spectrum is a sequence of pointed $G$-spaces $X_n$ with base point preserving $O(n)$ action and based $G$-maps $ \sigma_n : X_n \wedge S^1 \to X_{n+1}$  such that the iterated map 
$X_n \wedge S^{m}\to X_{m+n}$ is $O(n) \times O(m)$ equivariant. A morphism of $G$ orthogonal spectra $X$ and $Y$ are  $G$-map from $X_n \to Y_n$ which are compatible with the structure of an equivariant orthogonal spectrum.
Denote the category of equivariant orthogonal spectra by $\oI G\cS$. 
\end{defn}

Note this is not the original definition of equivariant orthogonal spectra. Let $\oI_G$ denote the $G$ topological category of real finite dimensional $G$-inner product spaces with the morphisms having a $G$-conjugation action. Let $\cTg$ denote the category of $G$-spaces with all morphisms and a $G$-conjugation on the morphism spaces.
\begin{defn} \cite{MM}
A $G$-orthogonal spectrum is a $G$ enriched functor from $X : \oI_G \to \cTg$ with $G$-equivariant maps $S^V \wedge X(W) \to X(V \oplus W)$ respecting the functoriality of $X$.  Morphisms are natural transformations respecting the suspension maps. Denote this category $\oI_G \cS$.  

\end{defn}

 As in the case of equivariant $\g$-spaces the two definitions are equivalent. There is a forgetful functor from $U: \oI_G \cS\to \oI G \cS$ and a  prolongation functor $E: \oI G \cS \to \oI_G \cS$ adjoint to each other. Refer to Schwede \cite{Schwede} 
 for the proof of the  equivalence of these categories 
 
 Define $G \oI \cS$ to be the category of orthogonal spectra where the morphisms are given by levelwise $G$-maps, that is, $G\oI \cS(X, Y)= \oI_G\cS(X,Y)^G$. This may confusing as we may also consider spectra in $G\oI[G\cT]$ and denote it by $G \oI \cS$ analogous to the $\g$-space case. We could have alternately proved that the category of spectra in $G\oI[G \cT]$ is a model category with appropriate structure. We can show that the inherited model structure on $\oI G \cS$ would be the same. 
 
 We will denote the prolongation functor from $\oI G \cS$ to $G \oI \cS$ also by $E$.

% We have the following diagram of categories  as in the case of equivariant $\g$-spaces.
%\begin{equation*}
%\xymatrix{ G \oI \cS \ar@<.005ex>[dr]^{\E'}  \ar@<-.1ex>[r]_{\U} & \oI G \cS \ar@<-.5ex>[l]_{\E} \ar@<.3ex>[d]^{\E' \circ \E} \\
%& \oI_G\cS \ar@<.6ex>[ul]^{U} \ar@<.5ex>[u]^-{\U}    }
%\end{equation*}
%The functor $U$ denotes the forgetful functor and $\E$, $\E'$ denote the prolongation functor. The functors $\E' \circ \E $ which we denote again by $\E$  and $\U$ define an equivalence of categories.  The functor $\E$ is defined as follows. 
 % Let $X$ be a $G\oI$-spectrum. Define the $\oI_G$-spectrum $\E X = \oI_G \otimes_{G\oI} X$ as follows.  Let $V$ be any indexing space in the given universe. Then 
 %\begin{small}
 %\begin{equation*}
%\xymatrix{\coprod \limits_{W',W} \oI_G(W, V)  \times G\oI(W', W) \times X(W') \ar@<.5ex>[r] \ar@<-.5ex>[r]  & \coprod \limits_{W} \oI_G (W.V) \times  X(W) \ar[r] & \E X(W) }
%\end{equation*}
%\end{small}
%It is straight forward to verify that $\E$ and $\ci$ are adjoints.

\begin{thm} \cite[Thm III.2.4]{MM}
The category $G \oI\cS$ is a compactly generated proper $G$-topological category with respect to level equivalences, level fibrations and q-cofibrations.
  \end{thm}
\begin{thm} \cite[Thm III.4.2]{MM}
 The category $G\oI\cS$ is a compactly generated proper $G$-topological model category with respect to $\pi_\ast$-isomorphisms, q-fibrations and q-cofibrations. The fibrant objects are the $\Omega$-$G$-spectra. 
 \end{thm}
 
 \begin{thm}
  Let $f: X \to Y$ be a morphism in $\oI G\cS$. Define 
 \begin{itemize}
 \item $f$ to be a weak equivalence (fibration) if $\E X \to\E Y$ is a weak equivalence is a $\pi_\ast$-isomorphism (q-fibration)
 \item and, $f$ to be a cofibration if it has the left lifting property with respect to acyclic cofibrations.
 \end{itemize}
 The category $\oI G \cS$ is a cofibrantly generated model category with weak equivalences, fibrations and cofibrations defined as above. Fibrant objects are $\Omega$-$G$-spectra. \end{thm}
\begin{proof}
Let $I$ and $J$ denote the generating cofibrations and acyclic cofibrations in $G \oI \cS$. It is easy to verify that that $UI$ and $UJ$ satisfy the cofibration hypothesis and rest of the proof follows from Theorem \ref{lem}.
\end{proof}

 \section{Comparison with equivariant orthogonal spectra }

 Given any $X: \g \to G \cT$ consider its prolongation to $\E X$ to $\cW$. Then $\E X(S^n)$ defines a spectrum. This is because given any functor $X : \cW \to G \cT $ we have a map $ \S^1 \to G\cT(\S^n, \S^{n+1}) \to G  \cT(X(\S^n), X(\S^{n+1}))$. Since $X$ is a functor there is a homomorphism from $O(n) \to \Aut(S^n) \to \Aut(\E X(S^n))$.  This defines an $O(n)$-action compatible with the spectrum $\E X$.  

Define the functor $\cB: \g[G \cT] \to \oI G \cS$ as $\cB X_n=\E X(\S^n)$.

%Let $\cU$ be a complete universe of $G$ real representations. 
%Evaluated at $\{A(S^V)\}_{V\in \cU}$, this defines a equivariant prespectrum $\cB A$. If $A$ is very special equivariant $\g$ space then this defines an equivariant $\Omega$-spectrum. 

Denote the sphere spectrum by $ \sS$.
Given any equivariant orthogonal spectrum $X$ we can define an equivariant $\g$-space $\cA X$  as follows.  Define $\cA L(\b n ) = \oI \cS(  \sS^n, L). $  to be the $G$-enriched morphism space of equivariant orthogonal spectra between $\sS^n$ and $X$.

%There is also an adjoint pair \begin{equation*}
%\xymatrix{ G\g[G\cT] \ar@<.5ex>[r]^{\cB_G} & G\oI\cS \ar@<.5ex>[l]^{\cA_G} }
%\end{equation*}
%defined in the same manner. The functor $\cB_G$ is defined using the   prolongation functor  and $\cA_G L (T)= \oI \cS (\sS_G^T, L)$ again being the $G$-enriched morphism space of equivariant orthogonal spectra.  Given an  equivariant orthogonal $\g$-space, $X : \g \to G\cT$ then $\E\cB X = \cB_G \E X$. 

Note that $\sS_G= \E \sS$ and the enriched morphism spaces  $\oI_G\cS(\E \sS, \E L)= \oI\cS(\sS, L)$ are equivalent for any equivariant orthogonal spectrum $L$.

%Thus we have the following adjoint pairs of functors with the commuting diagram. 
%\begin{equation*}
%\xymatrix{ \g[G\cT]  \ar@<-.5ex>[d]_{\E} \ar@<.5ex>[r]^{\cB} & \oI G\cS   \ar@<-.5ex>[d]_{\E} \ar@<.5ex>[l]^{\cA} \\
%G\g[G\cT] \ar@<-.5ex>[u]_{\P} \ar@<.5ex>[r]^{\cB_G} & G\oI\cS \ar@<-.5ex>[u]_{\P}\ar@<.5ex>[l]^{\cA_G}  }\end{equation*}
%where $\cS[\cTg] $ denotes the $G$-category of equivariant spectra.

If $ X \to Y$ is a level equivalence in $\oI G\cS$ then  $  \E X_V^H \to \E Y_V^H$ is a weak equivalence for all $H<G$ and $V \in G \oI$. Then  it is sufficient to verify that 
$$ H\oI  \cS(  \sS_G^T, \E X ) \to H\oI  \cS(\sS_G^T, \E Y)  $$ is a weak equivalence for all $H<G$ and $G$-sets $T$. This is true because given a finite $G$-CW complex $A$ and  a weak equivalence of  $B \to C$ of $G$-spaces, the map $\cT_G(A,B) \to \cT_G(A,C)$ is a weak equivalence of $G$-spaces.

%\begin{equation*}
%\xymatrix{ G\g[G\cT] \ar@<.5ex>[r]^{\cB_G} & G\oI\cS \ar@<.5ex>[l]^{\cA_G} }
%\end{equation*}
Let $X_V^H \to Y_V^H $ be level fibrations for all $H<G$ and $V \in G\cI$.  We need to show that the dotted arrow exists in the following diagram 
\begin{equation*}
\xymatrix{  (G/K  \times \D^n )^H\ar[r]\ar[d] & H\oI \cS(S_G^T, X) \ar[d]  \\
(G/K \times \D^n \times I)^H \ar@{-->}[ur] \ar[r] & H\oI \cS(S_G^T, Y) . }
\end{equation*}

By assumption  every $L <H$ dotted arrow exists in the following diagram of spaces. 

\begin{equation}\label{a}
\xymatrix{ (G/K  \times \D^n )^H \times \S_{GV}^T)^L \ar[r]\ar[d] & X_V^L \ar[d]  \\
((G/K \times \D^n \times I)^H  \times \S_{G V}^T)^L\ar@{-->}[ur] \ar[r] & Y _V^L . }
\end{equation}

Therefore the dotted arrow will exist in the following diagram by Elmendorf's theorem. 

\begin{equation*}
\xymatrix{(G/K  \times \D^n )^H\ar[r]\ar[d] & H \cT(S_{G V}^T, L_V) \ar[d]  \\
(G/K \times \D^n \times I)^H \ar@{-->}[ur] \ar[r] & H\cT(S_{GV}^T, K_V) . }
\end{equation*}

More generally if in diagram \ref{a} the maps are maps of spectra then the dotted arrow will also be a map of spectra. Hence we will get the required lift.

Therefore, the functor $\cB$ takes q-cofibrations to q-cofibrations.  Further $\cB$ takes stable equivalences to stable equivalences by construction.  This then implies that $\cA$ maps q-fibrations to q-fibrations. 
\begin{thm} \label{connective}
There exist adjoint pair of functors 
\begin{equation*}
\xymatrix{ \g[G\cT] \ar@<.5ex>[r]^{\cB} & \oI G\cS \ar@<.5ex>[l]^{\cA} }
\end{equation*}
This is a Quillen pair with the stable model structures on the two categories and is a "connective" Quillen equivalence. 
\end{thm}
\begin{proof}

From our previous discussion it is clear that $(\cA, \cB)$ are a Quillen pair.

To prove the theorem by \cite{MHovey}[Cor. 1.3.16] it is sufficient to show that  $\cB$ reflects weak equivalences between cofibrant objects and that the map $\cB Q\cA Y \to Y $ is a connective weak equivalence for all fibrant objects $Y$. 

If $\cB X \to \cB Y$ is a stable  weak equivalence of equivariant orthogonal spectra then by construction $X \to Y$ is a stable equivalence. 

Note for any equivariant orthogonal spectrum $Y$, $$\E \cA Y(T)= \oI_G \cS(\S_G^T, \E Y) \xrightarrow{\sim} \cT_G( T, \oI_G\cS(\S_G, \E Y) ).$$
Any fibrant equivariant orthogonal spectrum $Y$ is an equivariant $\Omega$-spectrum and hence $\cA Y$ is a very-special equivariant $\g$-space.  If $\cA X \to \cA Y$ is a stable equivalence for fibrant equivariant orthogonal spectra $X$ and $Y$ then this implies that $\cA X(\b 1) \to \cA Y(\b 1)$ is weakly equivalent, which implies $X\to Y$ is a stable equivalence. Hence $\cA$ reflects weak equivalences on fibrant objects. 

Let $Y$ be a fibrant equivariant orthogonal spectrum and taking a cofibrant-fibrant replacement in $\g[G\cT]$ we get a cofibrant very-special equivariant $\g$-space $C \cA Y$ stably equivalent to $\cA Y$.   
Note that $C \cA Y \to \cA Y $ is a stable equivalence implies that $ \E \cB C \cA Y \to \E \cB \cA Y$ is a stable equivalence. But $\cA Y$ is a very special equivariant space and hence $\cB\cA Y $ is an $\Omega$-spectrum. Note that $\cB\cA Y_0= \cA Y(\b 1) = \oI_G(\sS, Y)$. Since $Y$ is an $\Omega$-spectrum and $\cB \cA Y_0\to Y_0$ is a $G$-weak equivalence, $\cB \cA Y\to Y$  is a stable equivalence. Hence $\cB C \cA Y \to Y$ is a weak equivalence. 
Thus we have a connective Quillen equivalence.

\end{proof}

\subsection{Some remarks on $G$-symmetric monoidal categories}
\begin{defn}
Define a symmetric monoidal $G$-category to be a  category $\cC$ with a $G$-action and a symmetric monoidal structure $\otimes$ which commutes with the $G$-action.
\end{defn}

\begin{ex}
\begin{enumerate}
\item The category $\cF_G$ of finite $G$-sets with all set morphisms  with a trivial action on objects and a conjugation  $G$-action on the morphisms with disjoint union giving the symmetric monoidal structure.
\item Let $\cV_G$ ($\cW_G$) be the category with $G$-trivial discrete object space  of all real  (complex) $G$-representations and morphism $G$-space of all vector space morphisms  and a conjugation  $G$-action on the set of  morphisms. The $G$-symmetric monoidal structure is given by direct sum. 
\item Let $R$ be a commutative ring . Let $\cR_G$ denote the category of all $R[G]$-modules where the object space is discrete and has a trivial $G$-action and  the morphism space is the space of a all $R$-module homomorphisms with a $G$-conjugation action. The symmetric monoidal structure is given by direct sum.\end{enumerate}

\end{ex}
Segal \cite{Segal} gave a functor from the category of small symmetric monoidal categories to $\g$-spaces. The equivariant analogue was worked out by Shimakawa \cite{Shimakawa}.

Let $\Gcat$ denote the category symmetric monoidal $G$-catgeories. The $hat$ construction starts with a  symmetric monoidal $G$-category and constructs a $\g$-$G$-category. Then the functor $\h \cC: \g \to \Gcat$ defined as $\h \cC(\b n)= |\h\cC(n)|$ in \cite{Shimakawa}  defines a equivariant $\g$-space after applying the geometric realisation functor.  In particular, if the symmetry morphism in $\cC$ was identity then this will be a special $\g$-$G$-space.  

In general, however this is not the case and one way to  obtain special $\g$-$G$-spaces is to replace such a $G$-category by a   $G$-equivalent $G$-category. Shimakawa's main idea is to apply the hat construction to $\Hom(\u{EG}, \cC)$ where $\cC$ is a symmetric monoidal $G$-category and   $\u{EG} $ denotes the category with object set $G$ and a  unique morphism $\u{EG}(g_1,g_2)$ given by $g_2(g_1)^{-1}$. 

 The objects of the $G$-category $\text{Hom}_G(\u{EG}, \cC)$ are $G$-tuples  of isomorphic objects in  the category $\cC$. The $G$-action is via the $G$-action on $\u{EG}$ herefore, the action on the objects is via permutation of the $G$-tuples. The monoidal structure is  induced by the monoidal structure on $\cC$.

Let  $\cD$ be a  symmetric monoidal  category with a trivial $G$-action. Then the symmetric monoidal structure of $\cD$ induces a $G$-symmetric monoidal structure on $\cD_G $, the category of $G$-objects in $\cD$ with a conjugation action on the morphisms.  There is an adjoint pair between the category $\Hom (\u{EG}, \cD) $ and $\cD_G$ which is a lax $G$-symmetric monoidal functor. It is easy to verify that $\Hom(\u{EG}, \cD)$ is a special $\g$- $G$-category. Therefore, there is a equivariant stable equivalence between the two respective  equivariant $\g$-spaces obtained on applying the hat construction.  This gives an appropriate  special $\g$-$G$-category construction in examples mentioned earlier.

Guillou and May \cite{GuillouMay} also note that naive permuative $G$-categories will not give rise to equivariant $\E_\infty$-spaces. They define genuine permutative  $G$-categories and  show that any such category will give rise to an equivariant $\E_\infty$-space. 
This notion is equivalent to the idea of special $\g$-$G$-categories. A proof can be written down generalising  \cite{RS} for $G$-categories. This will be dealt with elsewhere.

\appendix
\section{}

{\sc Cofibration Hypothesis  \cite[5.3]{MMSS} :}  Let $\sA$ be a topological complete and cocomplete category with a continuous forgetful faithful functor to another topologically complete and cocomplete category $\cC$. Let tensors in $\cC$ be denoted by $X\wedge A$ and homotopies in particular defined using $X \wedge I_{+}$.  Let $I$ be a set of maps in $\sA$. The set $I$ is said to satisfy  the {\it cofibration hypothesis} if the following is true. 
\begin{enumerate}
\item Let $i:A \to B$ be a coproduct of maps in $I$ and $j$ be obtained by a cobase change from $i$ then $j$ is a h-cofibration.
\item Let $A$ be an object in $\sA$ which is a colimit of a sequence of  maps in $\sA$, which are all h-cofibrations when considered as maps in $\cC$. Then $A$ considered as an object of $\cC$ is a colimit of the corresponding maps in $\cC$.
\end{enumerate}

\begin{lemma}\cite[5.8]{MMSS} 
Let $I$ be a set of maps in $\sA$  such that each map in $I$ has a compact domain and $I$ satisfies the cofibration hypothesis. Then maps $f: X \to Y$ in $\sA$ factor functorially as composites 
$$ X \xrightarrow{i} X^\prime \xrightarrow{p} Y$$ such that $p$ satisfies the RLP with respect to any map in $I$ and $i$ satisfies the LLP with respect to any map that satisfies the RLP with respect to each map in $I$. Moreover, $i : X\to X' $ is a relative $I$-cell complex. 
\end{lemma}

Note a map $i: X \to X'$ is a relative $I$-cell complex if $X=Y_0$ and $i$ obtained by a colimit of a sequences of maps $Y_n \to Y_{n+1}$ where each of them is a obtained from a cobase change from a coproduct of maps in $I$.

\begin{thm}\label{lem}
Let $\sA$ and $\sB$ be topological complete and cocomplete categories and $\xymatrix{ \sA\ar@<.5ex>[r]^{F} & \sB \ar@<.5ex>[l]^{E} }$
  be an adjoint pair. Let $\sB$ be a cofibrantly generated model category with $I$ and $J$ as the generating cofibrations and trivial cofibrations.  Let $FI$ and $FJ$ satisfy the cofibration hypothesis. 

Then $\sA$ is a cofibrantly generated model category where $X \to Y$ is a weak equivalence or fibration if $E X \to EY$ is a weak equivalence of fibration in $\sB$ respectively and the cofibrations are all morphisms which have LLP with respect to all acyclic fibrations in $\sA$.
\end{thm}
\begin{proof}
The category $\sA$ is complete and cocomplete. The model category axioms namely two out of three, the property of retracts and one of the lifting properties will follow from definition of the model structure and adjointness of $E$ and $F$. 

The factorization theorem will follow from the small object argument. We need to prove that all  acyclic cofibrations in $\sA $ have a left lifting property with respect to all fibrations. 

It is then sufficient to show that every cofibration is a relative $FI$-cell complex. Let $A \to B$ be a trivial cofibration. Then by the small object argument it can be written as $A \xrightarrow{\sim} A' \xrightarrow{\sim} B$ where $A \to A'$ is a relative $FJ$-complex and $ A' \to B$ is a fibration. By our definition, $A \to B$ will have LLP with respect to $ A' \to B$ and therefore there exists a map $B \to A' $ which when composed with the map $A'\to B$ gives identity. 

Then $A \to B $ is a retract of the map $A \to A'$ and hence has a LLP with respect to all  fibrations in $\sA$. 

\end{proof}
\bibliographystyle{amsplain}

\begin{thebibliography}{99}
\bibitem{GuillouMay} Bertrand Guillou and J. P. May, \newblock \emph{Permutative $G$-categories in equivariant infinite loop space theory}. \newblock arXiv.1207.3459v2, math.AT. 2014.
\bibitem{MHovey} Mark Hovey. \newblock \emph{Model Categories} \newblock Mathematical Surveys and Monographs. American Mathematical Society, 63, 1999.
\bibitem{MM} M. A. Mandell and J. P. May. \newblock \emph{Equivariant Orthogonal spectra}. \newblock Memoirs of American Mathematical Society., 159, No.755, 2002. 
\bibitem{MMSS} M. A. Mandell, J.P. May, Stefan Schwede and  Brooke Shipley. \newblock \emph{Model categories of diagram spectra}. \newblock Proceedings of London Math. Soc., 82, No. 2 441-512. 2001.
\bibitem{Ostermayr} Dominik Ostermayr, \newblock\emph{Equivariant $\g$-spaces} \newblock arXiv:1404.7626, math.AT. 2014.
\bibitem{RS} Rekha Santhanam, \newblock \emph{Units of equivariant ring spectrum}. \newblock Algebraic and Geometric Topology., 11, No.3, 1361-1403, 2011.
\bibitem{Schwede} Stefan Schwede. \newblock \emph{Lecture notes on Equivariant Orthogonal spectra}. \newblock http://www.math.uni-bonn.de/people/schwede/
\bibitem{Segal} G Segal. \newblock\emph{Categories and generalized cohomology theories}. \newblock Topology, 13, 293-312, 1974.
\bibitem{Shimakawa} K Shimakawa. \newblock \emph{Infinite Loop $G$-spaces associated to Monoidal $G$-graded categories}. \newblock Publ. RIMS. Kyoto Univ., Vol 25, 239-262, 1989.
\bibitem{Shimakawa2} K Shimakawa, \newblock \emph{A note on $\Gamma_G$-spaces}. \newblock Osaka J. Math., 28. 223-228, 1991.
\end{thebibliography}

\end{document}